\documentclass[reqno,11pt,final]{amsart}

\title{A note about charts built by Eriksson-Bique and Soultanis on metric measure spaces}

\author{Luca Gennaioli, Nicola Gigli}
\address{Via Bonomea 256, SISSA}
\email{luca.gennaioli@sissa.it, nicola.gigli@sissa.it}

\usepackage[utf8]{inputenc}
\usepackage[english]{babel}
\usepackage{csquotes} 
\usepackage[a4paper,twoside]{geometry}
\usepackage{graphicx}
\usepackage[backend=biber,style=alphabetic,bibencoding=latin1,language=auto]{biblatex}

\usepackage{caption} 
\usepackage[hyperfootnotes=false]{hyperref} 
\hypersetup{
	colorlinks = true,
	linkcolor = {blue},
	urlcolor = {red},
	citecolor = {blue}
}
\usepackage{amsmath}
\usepackage{amsthm}
\usepackage{amssymb}
\usepackage{esint} 
\usepackage{mathrsfs} 
\usepackage{faktor} 
\usepackage{dsfont} 
\usepackage{xcolor}
\usepackage{array}
\usepackage{hhline}
\usepackage{enumitem} 
\usepackage{comment} 
\usepackage[toc,page]{appendix} 
\usepackage{imakeidx} 
\usepackage{xparse} 
\usepackage{mathtools}
\usepackage{fancyhdr} 
\usepackage{ifthen} 
\usepackage{forloop} 
\usepackage{xstring}
\usepackage{tikz}
\usetikzlibrary{babel} 
\usetikzlibrary{cd} 
\usepackage{emptypage} 

\usepackage[nameinlink,capitalise]{cleveref} 
\crefname{equation}{}{} 

\newcounter{results}[section] 

\theoremstyle{plain}
\newtheorem{theorem}[results]{Theorem}
\newtheorem{lemma}[results]{Lemma}
\newtheorem{proposition}[results]{Proposition}

\newtheorem*{theorem*}{Theorem}
\newtheorem*{lemma*}{Lemma}
\newtheorem*{proposition*}{Proposition}
\newtheorem*{corollary*}{Corollary}
\newtheorem*{exercise*}{Exercise}
\newtheorem*{fact*}{Fact}
\newtheorem*{problem*}{Problem}

\theoremstyle{remark}
\newtheorem{remark}[results]{Remark}

\newtheorem*{remark*}{Remark}
\newtheorem*{question*}{Question}

\theoremstyle{definition}
\newtheorem{definition}[results]{Definition}

\newtheorem*{definition*}{Definition}
\newtheorem*{example*}{Example}

\numberwithin{equation}{section}

\newcommand{\N}{\ensuremath{\mathbb N}} 
\newcommand{\R}{\ensuremath{\mathbb R}} 

\DeclarePairedDelimiter\norm{\lVert}{\rVert} 
\newcommand{\Leb}{\ensuremath{\mathscr L}} 
\newcommand{\essinf}{\text{ess}\inf} 

\newcommand{\de}{\ensuremath{\,\mathrm d}}

\ifdefined\comma 
	\renewcommand{\comma}{\ensuremath{\, \text{, }}}
\else 
	\newcommand{\comma}{\ensuremath{\, \text{, }}}
\fi

\newcommand{\Borel}{\ensuremath{\mathscr{B}}} 
\newcommand{\st}{\ensuremath{\ :\ }} 
\newcommand{\X}{\mathsf{X}} 
\newcommand{\Y}{\mathsf{Y}} 
\newcommand{\dist}{\text{d}} 
\newcommand{\Prphi}{{\rm Pr}}  
\newcommand{\Lebspace}{{\rm L}}  
\newcommand{\T}{{\rm T}}  

\newcommand{\m}{\ensuremath{\mathfrak m}}

\makeatletter
\newcommand{\newreptheorem}[2]{%
  \newtheorem*{rep@#1}{\rep@title}%
  \newenvironment{rep#1}[1]%
    {\def\rep@title{#2 \ref*{##1}}\begin{rep@#1}}%
    {\end{rep@#1}}%
}%
\makeatother
\theoremstyle{plain}
\newreptheorem{theorem}{Theorem}

\begin{document}

\begin{abstract}
This note is motivated by recent studies by Eriksson-Bique and Soultanis about the construction of charts in general metric measure spaces. We analyze their construction and provide an alternative and simpler proof of the fact that these charts exist on sets of finite Hausdorff dimension. The observation made here offers also some simplification about the study of the relation between the reference measure and the charts in the setting of $\text{RCD}$ spaces.
\end{abstract}

\maketitle
\tableofcontents
\thispagestyle{empty}

\section{Introduction}
In the recent, very interesting, paper \cite{EBS21} the authors provided a general construction of charts on metric measure spaces, key features of their notion being: the compatibility with Sobolev calculus (and thus in particular with the differential calculus as developed by Cheeger in \cite{Cheeger00} and the second author \cite{Gigli12}), a very general existence result, notable consequences in terms of the structure of the Sobolev spaces (see also \cite{EBRS22} and \cite{EBRS222}). An example in this latter direction is the proof that the space $W^{1,p}(\X)$, $p\in(1,\infty)$, is reflexive as soon as the space $\X$ can be covered by a countable number of sets with finite Hausdorff measure (the `previous best’ result appeared in \cite{ACM14} and required the metric to be locally doubling).

A crucial step in \cite{EBS21} is the proof that if $\varphi:E\subset\X\to\R^n$ is a `$p$-independent weak chart’, then $n$ is bounded from above by the Hausdorff dimension of $E$. For the precise meaning of `$p$-independent weak chart’ we refer to $\ref{Ebs chart}$; for the purpose of this introduction we shall limit ourselves to point out that in the smooth setting this would be equivalent to requiring the image of the differential of $\varphi$ at every point to span the whole tangent space of $\R^d$. Starting from this result, existence of actual charts is obtained via a suitable maximality argument.

Interestingly, this upper bound is proved via means that have, in principle, little to do with analysis in non-smooth setting: key ingredients are indeed the elliptic regularity result in \cite{DPR} and the study of the structure of the set of non-differentiability points of Lipschitz functions in \cite{AM16}.

This sort of procedure has a recent analogue in the theory of $\text{RCD}$ spaces. Let us recall indeed that in \cite{Mondino-Naber14} it has been proved that finite dimensional $\text{RCD}$ spaces admit bi-Lipschitz charts covering almost all the space. In \cite{Mondino-Naber14} no information about the behaviour of the reference measure w.r.t. these charts has been provided: this topic has been later studied in \cite{MK16}, \cite{DPMR16}, \cite{GP162} where, relying in a way or another on \cite{DPR} and \cite{AM16}, it has been proved that $\varphi_*(\m_{|E})\ll\mathcal \Leb^n$ for a Mondino-Naber chart $\varphi:E\to\R^n$.

Of particular interest for the kind of discussion we want to make here is the fact that in \cite{GP162} only the results in \cite{DPR} have been used, while in \cite{MK16} also those in \cite{AM16} were necessary. Comparing this with the results in \cite{EBS21} it is natural to wonder whether the use of \cite{AM16} is really crucial or can be avoided: this is the question motivating the present note. Of course, there is nothing wrong in using a well-established result in doing research, our study is simply motivated by the desire of better understanding the interesting construction done in \cite{EBS21}. The result of our investigation is that \cite{AM16} is not really needed and the line of thought presented here simplifies not only some of the steps done in \cite{EBS21}, but also some of those in \cite{GP162}: see Section $\ref{Main result section}$.

Another remark that we make, consequence of the studies in \cite{EBS21}, is that the dimension of the (co)tangent module (in the sense of \cite{Gigli12}) on a subset $E\subset\X$ is bounded from above from the Hausdorff dimension of $E$, see Remark $\ref{Bdd dimension tangent}$.

\emph{Acknowledgments}: We wish to thank Elefterios Soultanis for the numerous conversation we had with him while working on this manuscript.

\section{Preliminaries}

In this section we shall recall the definition of Sobolev space following the approach in \cite{AmbrosioGigliSavare11}. 
We say that a triple $(\X,\dist,\m)$ is a metric measure space if $(\X,\dist)$ is a complete and separable metric measure space and $\m$ is a Radon measure which is finite on balls. For the rest of the paper $p,q$ will be conjugate exponents, namely $\frac{1}{p}+\frac{1}{q}=1$.
\begin{definition}
We say that a probability measure $\pi$ on $C([0,1];\X)$ is a $\emph{q-test plan}$ if it is concentrated on $AC([0,1];\X)$ and the following two conditions are met:
\begin{enumerate}
    \item $\exists$ $C=C(\pi)>0$ such that $e_{t\sharp}\pi\leq C\m$, where $\m$ is the reference measure on $\X$ and $e_t:C([0,1];\X)\to\X$ is the evaluation map $e_t(\gamma)=\gamma_t$. 
    \item The following quantity, called $\emph{kinetic energy}$, is finite
    \begin{equation*}
        \text{K.E.}(\pi)=\int\int_0^1|\dot{\gamma}_t|^q\de t\de\pi(\gamma),
    \end{equation*}
with $|\dot{\gamma}_t|=\lim_{h\to 0}\frac{\de(\gamma_{t+h},\gamma_t)}{h}$ is the metric derivative of the curve $\gamma$.
\end{enumerate}
\end{definition}
With this notion at hand we can introduce the Sobolev space $\text{W}^{1,p}(\X,\de,\m)$:
\begin{definition}
We say that a function $f:\X\to\R$ belongs to the Sobolev space $W^{1,p}(\X,\de,\m)$ if $f\in \Lebspace^p(\m)$ and if
\begin{equation}
\label{Def of Sobolev}
    \int|f(\gamma_1)-f(\gamma_0)|\de\pi(\gamma)\leq\int\int_0^1G(\gamma_t)|\dot{\gamma}_t|\de t\de\pi(\gamma)\quad\forall\pi\;\text{q-test\;plan},
\end{equation}
with $G:\X\to\R_+$ being a Borel function belonging to $\Lebspace^p(\m)$.
\end{definition}
\begin{remark}
It is easy to see that the set of functions $G$ satisfying $\ref{Def of Sobolev}$ is a closed convex set, hence it admits an element of minimal norm: we will call such an element $\emph{p-weak upper gradient}$ and we will denote it by $|Df|_p$. With a little bit of work it is possible to prove that the function $|Df|_p$ is such that $|Df|_p\leq G$ $\m$-a.e. for every other $G$ satisfying $\ref{Def of Sobolev}$.
\end{remark}

We now switch our attention to the theory of $\Lebspace^p(\m)$-normed $\rm{L}^\infty(\m)$-modules that the second author built in \cite{Gigli14}: the following material can be found there, unless otherwise stated.

\begin{definition}[$\Lebspace^p(\m)$-normed module]
We say that a Banach space $(\mathcal{M},\norm{\cdot}_\mathcal{M})$ is an $\Lebspace^p(\m)$-normed $\rm{L}^\infty(\m)$-module if there exists a bilinear continuous map $\cdot:\rm{L}^\infty(\m)\times\mathcal{M}\to\mathcal{M}$ which makes $\mathcal{M}$ a module with unity over the ring of $\rm{L}^\infty(\m)$ functions  and another map $|\cdot|:\mathcal{M}\longrightarrow \Lebspace^p(\m)$ with nonnegative values such that 
\begin{align}
    \label{Pointwise norm}
    \norm{|v|}_{\Lebspace^p(\m)}&=\norm{v}_{\mathcal{M}}, \\
    |f\cdot v|&=|f||v|\qquad\m-\text{a.e.}
\end{align}
for all $v\in\mathcal{M}$, $f\in \rm{L}^\infty(\m)$. We call $\cdot$ the multiplication and $|\cdot|$ the $\emph{pointwise norm}$.
\end{definition}
\begin{remark}
Note that the pointwise norm is continuos thanks to the triangular inequality, in fact
\begin{equation*}
    \norm{|v|-|w|}_{\Lebspace^p(\m)}\leq\norm{|v-w|}_{\Lebspace^p(\m)}=\norm{v-w}_\mathcal{M}.
\end{equation*}
Moreover with a little bit of abuse of notation we will write $fv$ instead of $f\cdot v$ and write $\Lebspace^p(\m)$-normed module instead of $\Lebspace^p(\m)$-normed $\rm{L}^\infty(\m)$-module.
\end{remark}
A related interesting concept is the one of $\emph{localization}$ of a module, indeed it is easy to see that the following object
\begin{equation*}
    \mathcal{M}_{|E}:=\{\chi_E v:\;v\in\mathcal{M}\}
\end{equation*}
is a submodule of $\mathcal{M}$ and it clearly inherits the normed structure from $\mathcal{M}$.
\begin{definition}[$\textbf{Local independence}$]
Let $\mathcal{M}$ be an $\Lebspace^p(\m)$-normed $\rm{L}^\infty(\m)$-module and $A\in\Borel(X)$ with $\m(A)>0$, we say that a family $v_1,...,v_n\in\mathcal{M}$ is independent on $A$ if for every $f_1,...,f_n\in \rm{L}^\infty(\m)$
\begin{equation}
\label{Local independence}
    \sum_{i=1}^n f_i v_i = 0\quad\m-\text{a.e.}\;\text{on}\;A\implies f_i=0\quad\m-\text{a.e.}\;\text{on}\;A\quad\forall i=1,...,n.
\end{equation}
\end{definition}
In the spirit of linear algebra we shall also define what is the $\emph{span}$ of a set of vectors
\begin{definition}[$\textbf{Span}$]
Let $\mathcal{M}$ be an $\Lebspace^p(\m)$-normed $\rm{L}^\infty(\m)$-module, $V\subset\mathcal{M}$ a subset and $A\in\Borel(X)$. We denote with $\text{Span}_A(V)$ the closure in $\mathcal{M}$ of the $\rm{L}^\infty(\m)$-linear combinations of elements of $V$.
Moreover we say that $\text{Span}_A(V)$ is the space generated by $V$ on $A$.
\end{definition}

After this natural definition, the one of basis and of dimension for an $\Lebspace^p(\m)$-normed $\rm{L}^\infty(\m)$ arise naturally:
\begin{definition}
We say that a finite family $v_1,...,v_n\in\mathcal{M}$ is a basis on $A\in\Borel(X)$ if it is independent on $A$ and $\text{Span}_A\{v_1,...,v_n\}=\mathcal{M}_{|A}$. If the above happens we say that the $\emph{local dimension}$ of $\mathcal{M}$ on $A$ is $n$ and in case $\mathcal{M}$ has not dimension $k$ for any $k\in\N$ we say that it has infinite dimension.
\end{definition}
It can be proved that the notion of dimension is well-posed, namely if we have $v_1,...,v_n$ generating $\mathcal{M}$ on a set $A$ and $w_1,...,w_m$ are independent on $A$, then $n\geq m$. Ultimately this means that two different basis must have the same cardinality.

Building over these tools we have the following proposition:
\begin{proposition}
Let $\mathcal{M}$ be an $\Lebspace^p(\m)$-normed $\rm{L}^\infty(\m)$-module. Then there is a unique partition $\{E_i\}_{i\in\N\cup\{\infty\}}$ of $\X$, up to $\m$-a.e. equality, such that:
\begin{enumerate}
    \item for every $i\in\N$ such that $\m(E_i)>0$, $\mathcal{M}$ has dimension $i$ on $E_i$, \\
    \item for every $E\subset E_\infty$ with $\m(E)>0$, $\mathcal{M}$ has infinite dimension on $E$.
\end{enumerate}
\end{proposition}

We now introduce the notion of $\emph{pullback module}$ which, roughly speaking, is nothing but a module over a space $\X$ obtained by pulling back
a module on another space $\Y$ via a certain map.

\begin{definition}[Pullback]
Let  $(\X,\dist_\X,\m_\X)$ and $(\Y,\dist_\Y,\m_\Y)$ be metric measure spaces, $\varphi: \X\longrightarrow \Y$ a map of bounded compression and $\mathcal{M}$ and $\Lebspace^p(\m_\Y)$-normed module. Then there exists a unique, up to unique isomorphism, couple $(\varphi^*\mathcal{M},\varphi^*)$ with $\varphi^*\mathcal{M}$ being an $\Lebspace^p(\m_\X)$-normed module and $\varphi^*:\mathcal{M}\longrightarrow\varphi^*\mathcal{M}$ being a linear and continuous operator such that:
\newline
\begin{enumerate}
    \item $|\varphi^*v|=|v|\circ\varphi$ holds $\m_\X$-a.e., for every $v\in\mathcal{M}$, \\
    \item the set $\{\varphi^*v\st v\in\mathcal{M}\}$ generates $\varphi^*\mathcal{M}$ as a module.
\end{enumerate}

\end{definition}
At this point one can try to understand what is the relation between the dimension of a module and the one of its pullback via the map $\varphi$ and in order to do so we need to introduce a sort of $\emph{left inverse}$ of the pullback operator $\varphi^*$. To do so let us assume $\varphi_\sharp\m_\X=\m_Y$ to simplify the exposition.

For $f\in \Lebspace^p(\m_\X)$ nonnegative we put
\begin{equation}
\label{Projection map}
\Prphi_\varphi(f):=\frac{\de\varphi_\sharp(f\m_\X)}{\de\m_\Y}
\end{equation}
and in a natural way we set $\Prphi_\varphi(f):=\Prphi_\varphi(f^+)-\Prphi_\varphi(f^-)$ for general $f\in \Lebspace^p(\m_\X)$. 

For the next proposition we need to recall the classical Disintegration theorem. The statement below is taken from \cite[Theorem 5.3.1]{AmbrosioGigliSavare08}, see also \cite[Chapter 452]{Fremlin4} and \cite[Chapter 10.6]{Bog07}

\begin{theorem}[Disintegration]
Let $\X,\Y$ be complete and separable metric spaces, $\mu\in\mathcal{P}(\X)$, let $\pi:\X\to\Y$ be a Borel map and let $\nu=\pi_\sharp\mu\in\mathcal{P}(\Y)$. Then there
exists a $\nu$-a.e. uniquely determined Borel family of probability measures $\{\mu_y\}_{y\in\Y}\subseteq\mathcal{P}(\X)$  such that
$\mu_x(\X\setminus\pi^{-1}(\{y\}))=0$ for $\nu$-a.e. $y\in\X$ and 
\begin{equation}
    \label{disintegration}
    \int_{\X}f\de\mu = \int_\Y\biggl(\int_{\pi^{-1}(\{y\})}f\de\mu_y\biggr)\de\nu(y)
\end{equation}
for every Borel map $f:\X\to[0,+\infty]$.

\end{theorem}
\begin{remark}
Two remarks are in order here: the first one is that the above theorem in \cite{AmbrosioGigliSavare08} is stated for Radon separable metric space but in our setting it suffices to state it for complete and separable ones (which in particular are Radon), the second is that the result easily extends to any $f:\X\to\R$ Borel provided for example that $f\in L^1(\mu)$.
\end{remark}
We now recall some properties of the map $\Pr_\varphi$.
\begin{proposition}
The operator $\Prphi_\varphi: \Lebspace^p(\m_\X)\longrightarrow \Lebspace^p(\m_Y)$ is linear, continuous and 
\begin{equation}
    \label{RN and disintegration}
    \Prphi_\varphi(f)(y) = \int_\X f(x)\de\m_y(x)\quad\m_\Y-\text{a.e.},\quad\forall f\in \Lebspace^p(\m_\X),
\end{equation}
where $y\mapsto m_y$ denotes the disintegration of $\m_\X$ with respect to the map $\varphi$. Finally it holds
\begin{equation}
    \label{Projection contraction}
    |\Prphi_\varphi(f)|\leq\Prphi_\varphi(|f|)\quad\m_\Y-\text{a.e.}
\end{equation}
\end{proposition}
\begin{proof}
Linearity is a consequence of the linearity of the integral. Formula ($\ref{Projection contraction}$) is also trivial while for ($\ref{RN and disintegration}$) we have for any $A\in\Borel(\Y)$
\begin{equation*}
    \int_{A}\Prphi_\varphi(f)(y)\de\m_\Y = \int_A\de\varphi_\sharp(f\de\m_\X)=\int_{\varphi^{-1}(A)}f(x)\de\m_\X,
\end{equation*}
and by the properties of the disintegration we have
\begin{equation*}
\int_{\varphi^{-1}(A)}f(x)\de\m_\X=\int_\Y\int_{\varphi^{-1}(A)}f(x)\de\m_y(x)\de\m_\Y(y) = \int_A\int_{\X}f(x)\de\m_y(x)\de\m_\Y(y),
\end{equation*}
therefore proving ($\ref{RN and disintegration}$). 

To prove continuity note that the case $p=\infty$ is due to formula ($\ref{Projection contraction}$) while continuity in $\Lebspace^p(\m)$ for every $p\in[1,+\infty)$ follows from the following 
\begin{equation*}
\int_{\Y}|\Prphi_{\varphi}|^p\de\m_{\Y}=\int_{\Y}\biggl|\int_{\X}f(x)\de\m_y(x)\biggr|^p\de\m_{\Y}(y)\leq\int_{\Y}\int_{\X}|f(x)|^p\de\m_y(x)\de\m_{\Y}(y)=\norm{f}^p_{\Lebspace^p(\m)},
\end{equation*}
where we used Jensen's inequality and the properties of the disintegration.
\end{proof}

In the case of a general $\Lebspace^p(\m_\X)$-normed module the continuous operator $\Prphi_\varphi:\varphi^*\mathcal{M}:\longrightarrow\mathcal{M}$ can be characterized by the following properties:
\begin{align}
    \label{prop of projection}
    g\Prphi_\varphi(v)=\Prphi_\varphi(g\circ\varphi v),\quad\forall v\in\mathcal{M}\quad\forall g\in \rm{L}^\infty(\m_\X) \\ 
    \Prphi_\varphi(g\varphi^*v) = \Prphi_\varphi(g)v\quad\forall v\in\mathcal{M}\quad\forall g\in \rm{L}^\infty(\m_\X),
\end{align}
with the bound $|\Prphi_\varphi(V)|\leq\Prphi_\varphi(|V|)$ still holding $\m_\Y$-a.e. for every $V\in\varphi^*\mathcal{M}$.

With these objects we are now able to describe the structure of the pullback module, in particular (as one can expect by reasoning via pre-composition) the pullback of an $n$-dimensional module $\mathcal{M}$ over $E$ is an $n$-dimensional module over $\varphi^{-1}(E)$ (see also \cite{PE18}).

\begin{proposition}
Let $\mathcal{M}$ be an $\Lebspace^p(\m_\Y)$-normed module over the m.m.s. $(\Y,d_\Y,\mu)$ and let $E\in\mathcal{B}(Y)$ be a Borel set where $\mathcal{M}$ has dimension $n$, with $\{v_1,...,v_n\}$ being a basis. Let $(\X,d_\X,\m)$ be another m.m.s. and $\varphi:\X\to \Y$ be a map of bounded compression such that $\varphi_{\sharp}\m_\X = \m_\Y$, then $\{\varphi^{*}v_1,...,\varphi^{*}v_n\}$ is a basis of $\varphi^{*}\mathcal{M}$ over $\varphi^{-1}(E)$. 
\end{proposition}
\begin{proof}
We first prove that $\{\varphi^*v_1,...,\varphi^*v_n\}$ generate $\varphi^*\mathcal{M}$ over $\varphi^{-1}(E)$.

First recall that $\varphi^*\mathcal{M}$ is generated (as module) by $\{\varphi^{*}v: v\in\mathcal{M}\}=:V$. Let us show that $V\subseteq\text{Span}_{\varphi^{-1}(E)}\{\varphi^{*}v_1,...,\varphi^{*}v_n\}$: pick $w\in V$, then there exists $v\in\mathcal{M}$ such that $w=\varphi^{*}v$ so that there exists $(A_j)_j\subseteq\Borel(\X)$ partition of $E$ and $(g_i^j)_{j\in\N}\subset \rm{L}^\infty(\m_\Y)$ $\forall i=1,...,n$ such that 
\begin{equation*}
    \chi_{A_j}v = \sum_{i=1}^n g_i^j v_i\qquad\forall j\in\N
\end{equation*}
Using the linearity of the pullback map and the fact that $\varphi^*(g v)=g\circ\varphi\varphi^{*}v$ for all $v\in\mathcal{M}$, $g\in \rm{L}^\infty(\m_\Y)$  we get
\begin{equation*}
    \chi_{\varphi^{-1}(A_j)}w = \sum_{i=1}^n g_i^j\circ\varphi\varphi^{*}v_i.
\end{equation*}
Finally, since the pullback module has a natural structure of $\Lebspace^p(\m)$-normed $\rm{L}^\infty(\m)$-module, we get that $\text{Span}_{\varphi^{-1}(E)}\{\varphi^{*}v_1,...,\varphi^{*}v_n\}$ is closed, proving the first result.

We now turn to local independence: assume by contradiction $\{\varphi^*v_1,...,\varphi^{*}v_n\}$ are not independent on $\varphi^{-1}(E)$ then there exist $f_1,...,f_n\in \rm{L}^\infty(\m_\X)$ such that $\sum_{i=1}^n f_i\varphi^{*}v_i=0$ $\m$-a.e. with (upon relabeling indexes) $|f_1|>0$ $\m$-a.e. on some subset $\tilde{E}$ of positive measure. Without loss of generality, possibly considering a smaller set, we shall assume $f_1>0$ $\m$-a.e. so that
\begin{equation*}
    \sum_{i=1}^n f_i\varphi^*v_i=0\quad\m-\text{a.e. on}\;\tilde{E}\implies\sum_{i=1}^n \text{Pr}_\varphi(f_i) v_i=0\quad\m-\text{a.e. on}\;\tilde{E}.
\end{equation*}
However note that $\text{Pr}_{\varphi}(f_1)>0 $ on some set of positive $\m_\Y$ measure, contradicting the independence of the $v_i$s.

\end{proof}

\begin{remark}
We stress again that the assumption $\varphi_\sharp\m_\X=\m_\Y$ is purely for the sake of exposition.
\end{remark}

\begin{definition}
We say that the space of $\rm{L}^\infty(\m)$-linear and continuous maps $L:\mathcal{M}\to L^1(\m)$ is the dual module of the module $\mathcal{M}$ and we shall denote this space by $\mathcal{M}^*$.
\end{definition}
\begin{remark}
Being $\mathcal{M}$ $\Lebspace^p(\m)$-normed, we can endow $\mathcal{M}^*$ with a natural structure of $\Lebspace^q(\m)$-normed module.
\end{remark}
We are now in position to speak about the differential of a Sobolev function as the following proposition shows.
\begin{proposition}
\label{Differential module}
Let $(\X,\de,\m)$ be a metric measure space, then there exists a unique (up to isomorphism) couple $(\Lebspace^p(\T^{*}\X),\de_p)$ where $\Lebspace^p(\T^{*}\X)$ is an $\Lebspace^p(\m)$-normed $\rm{L}^\infty(\m)$-module and $\de_p: W^{1,p}(\X)\to \Lebspace^p(\T^{*}\X)$ is a linear and continuous operator such that:
\begin{enumerate}
    \item $|\de_p f|=|Df|_p$ $\m$-a.e. for every $f\in W^{1,p}(\X)$,
    \item The set $\{\de f:\; f\in W^{1,p}(\X)\}$ generates $\Lebspace^p(\T^{*}\X)$.
\end{enumerate}
\end{proposition}
\begin{remark}
We will call \emph{1-forms} the elements of $\Lebspace^p(\T^{*}\X)$, in analogy with the section of the cotangent bundle on a Riemannian manifold. 
\end{remark}
\begin{definition}
We denote with $\Lebspace^q(\T\X)$ the dual module of $\Lebspace^p(\T^{*}\X)$ and we call its elements vector fields or vectors.
\end{definition}

Besides the differential of a Sobolev function introduced in $\ref{Differential module}$, one can give another definition which exploits the fact that the map is Lipschitz and of bounded compression: this class of maps is that of $\emph{bounded deformation}$. In this direction we need to recall the notion of $\emph{pullback of forms}$: in order to distinguish it from the pullback of a module we shall proceed denoting with $\omega\mapsto[\varphi^*\omega]$ the pullback map and with $\varphi^*$ the pullback of 1-forms which is the following:
\begin{definition}
Let $\varphi:\X\to\Y$ be a map of bounded deformation, then we define $\varphi^*:\Lebspace^p(\T^*\Y)\to \Lebspace^p(\T^{*}\X)$ to be the linear map such that $\varphi^*(\de f)=\de(f\circ\varphi)$ for all $f\in W^{1,p}(\Y)$ and $\varphi^*(g\omega)=g\circ\varphi\varphi^*\omega$ for all $g\in \rm{L}^\infty(\Y)$ and $\omega\in \Lebspace^p(\T^*\Y)$.
\end{definition}
\begin{remark}
It is easy to see that, thanks to the regularity properties of $\varphi$, the pullback of 1-forms $\varphi^*$ is well defined.
\end{remark}
\begin{definition}
\label{Diff of bdd deformation}
\begin{equation}
    \label{BD differential}
    [\varphi^*\omega](\underline{\de_p\varphi}(v))=\varphi^*\omega(v)\quad\forall v\in \Lebspace^q(\T\X),\;\;\forall\omega\in \Lebspace^p(\T^*\Y).
\end{equation}
\end{definition}

In the recent work \cite{EBS21} the authors provide some ``charts" over Borel sets $(E_i)_{i\in\N}$ partitioning the metric measure space $\m$-a.e.: we will briefly recall here the definition
\begin{definition}
\label{Ebs chart}
We say $\varphi:\X\to\R^N$ is an EBS chart over the Borel set $E$ if it is a Lipschitz map with the following properties
\begin{enumerate}
    \item (p-independence) $\essinf_{v\in\mathbb{S}^{N-1}}|D(v\cdot\varphi)|_p>0$ $\m$-a.e on $E$.
    \item (maximality) There is no other Lipschitz map $\varphi:\X\to\R^M$ with $M>N$ which is p-independent on a subset of $E$ of positive measure.
\end{enumerate}
\end{definition}
The authors proved that the condition of p-independence over a set $E$ is equivalent to the fact that the $\Lebspace^p(\T^{*}\X)$ module over $E$ is generated by the differentials of the components of the chart: in other words $\{\de\varphi^1,...,\de\varphi^N\}$ is a basis for $\Lebspace^p(\T^{*}\X)_{|E}$ (see Lemma 6.3 in \cite{EBS21}) and as a consequence of Theorem 1.4.7 in \cite{Gigli14} we are able to deduce that $\Lebspace^q(\T\X)_{|E}$ is also an $N$-dimensional normed module.

\section{Main result}
\label{Main result section}
In this section we give an alternative proof to Proposition 4.13 in \cite{EBS21}. First we remark that with $\underline{\de_p\varphi}$ we will denote the differential of a map of bounded deformation in the sense of definition $\ref{Diff of bdd deformation}$, while with $\de_p f$ we denote the differential in the sense of Proposition $\ref{Differential module}$. Lastly let us assume that $\m$ is a finite measure: we can do so because of the inner regularity of the measure $\m$. Indeed if for a Borel map $\psi:\X\to\R^n$ we have $\psi_\sharp(\m_{|E_k})<<\Leb^n$ for every $k\in\N$ with $(E_k)_k$ compact, such that $E_{k}\subseteq E_{k+1}$ and $\m(E\setminus\cup_k E_k)=0$, then $\psi_\sharp(\m_{|E})<<\Leb^n$.

We begin with the following simple lemma which follows standard arguments in linear algebra:
\begin{lemma}
\label{Algebraic Lemma}
Let $\mathcal{M}$ be an $\Lebspace^p(\m)$-normed module and $\mathcal{M}^{*}$ be its dual module. Assume that $\mathcal{M}$ has dimension $n$ over $E$: then $\{v_1,...,v_n\}$ and $\{\omega_1,...,\omega_n\}$ are basis of $\mathcal{M}^*$ and $\mathcal{M}$ (respectively) over $E$ if and only if $\text{det}[ \omega_i(v_j)]_{ij}>0$ $\m$-a.e. on $E$.
\end{lemma}
\begin{proof}
Define $A_{ij}:=[ \omega_i(v_j)]_{ij}$ and let us assume first that $\det A>0$ $\m$-a.e.. It is clearly sufficient to prove the independence: assume by contradiction that $\sum_{i=1}^n g_iv_i=0$ $\m$-a.e. on some subset $B$ of positive measure, for some $g_1,...,g_n$ which are not all zero on $B$ (in the measure theoretic sense). Then consider $\underbar{g}:=(g_1,...,g_n)$ and note that $A\underbar{g}\neq 0$ $\m$-a.e. on $B$ because of the condition on the determinant. However $A\underbar{g}=\sum_{i=1}^n g_iv_i(\underbar{$\omega$})$=0 $\m$-a.e. on $B$, which is clearly a contradiction. This argument trivially applies for $\{\omega_1,...,\omega_n\}$ as well by considering the transpose of $A$.

Assume now that $\{\omega_1,...,\omega_n\}$ and $\{v_1,...,v_n\}$ are basis over $E$ of $\mathcal{M}$ and $\mathcal{M}^*$ respectively and by contradiction let $\det A = 0$ $\m$-a.e. on a Borel subset $C$ of positive measure. Then there exists a further measurable subset (which we won't relabel) $C$ of positive measure and $\underbar{g}\in \rm{L}^\infty(\m)^n$ for which $A\underbar{g}=0$ and $\underbar{g}\neq 0$ $\m$-a.e. on $C$. The latter system of equations means that  we have
\begin{equation}
\label{indep assum}
    v_i\biggl(\sum_{j=1}^n g_j\omega_j\biggr) = 0\quad\m-\text{a.e.}\; \text{on}\; C,\;\forall i=1,...,n.
\end{equation}
Set $\tilde{\omega}=\sum_{j=1}^n g_j\omega_j$ and suppose that $|\tilde{\omega}|\neq 0$ $\m$-a.e. on $C$, then there exists a non-zero continuous functional $\ell\in\mathcal{M}^\prime$ (which is the Banach dual) such that $\ell(\chi_C\tilde{\omega})=||\chi_C\tilde{\omega}||_\mathcal{M}$ and there exists $L\in\mathcal{M}^*$ (see Proposition 1.2.13 in \cite{Gigli14}) such that 
\begin{equation*}
    \ell(\omega) = \int_\X L(\omega)\de\m\quad\forall\omega\in\mathcal{M}.
\end{equation*}
In our case this means that $\norm{\chi_C\tilde{\omega}}_\mathcal{M}=\int_{C}L(\tilde{\omega})\de\m > 0$, so that there must be a Borel set of positive measure where $\chi_C L(\tilde{\omega})>0$, which contradicts ($\ref{indep assum}$) since there exists $D\subset C$ with $\m(D)>0$ such that $\chi_D L=\sum_{i=1}^n f_i v_i$ for some $f_1,...,f_n\in \rm{L}^\infty(\m)$. 
\end{proof}

\begin{lemma}
\label{Pullback and cotangent}
Let $\varphi$ be an EBS chart over the Borel set $E$ and $\{v_1,...,v_n\}\in \Lebspace^p(\T\X)$ be independent over $E$, then $\{\underline{\de_p\varphi}(v_1),...,\underline{\de_p\varphi}(v_n)\}\in\varphi^*\Lebspace^p_\mu(T\R^n)$ are independent over the same set.
\end{lemma}
\begin{proof}
Consider $f_1,...,f_n\in \rm{L}^\infty(\m)$ such that 
\begin{equation*}
    \sum_{i=1}^n f_i\underline{\de_p\varphi}(v_i)=0\quad\m-\text{a.e.}\;\;\text{on}\;\; E,
\end{equation*}
then set $v:=\sum_{i=1}^n f_iv_i$. Note that the maps $\Pi^j:\R^n\longrightarrow\R$ being the projection on the $j$-th component are all 1-Lipschitz with respect to the Euclidean distance and for this reason they belong to $W^{1,p}(\R^n,\text{d}_{\text{eucl}},\mu)$: following equation (\ref{BD differential}) we have that, for every $j=1,...,n$ and choosing $\omega=\de_p\Pi_j$,  
\begin{equation*}
    0 = \de\varphi^j(v)=\sum_{i=1}^n f_i\de\varphi^j(v_i)\quad\m-\text{a.e.}\;\;\text{on}\;\; E.
\end{equation*}
 Being the matrix $A=(A_{ij})_{ij}=\langle\de\varphi^j,v_i\rangle$ such that $\det A>0$ $\m$-a.e., the equations above can be rewritten as $A\underbar{f} = 0$ $\m-$a.e. on $E$ with $\underbar{f}=(f_1,...,f_n)$, meaning $\underbar{f}=0$ thanks to Lemma \ref{Algebraic Lemma}.
\end{proof}

The following result is borrowed from \cite{LPR21} (Proposition 4.5) where only the metric measure space $(\R^n,\de_{\text{eucl}},\mu)$ is considered.

\begin{proposition}
\label{Pasqualetto Lucic Rajala}
Assume that there exists a Borel set $E$ such that $\dim\Lebspace^p_\mu(\T^*\R^n)_{|E}=n$ for some $p\in (1,+\infty)$, then $\mu_{|E}<<\Leb^n$.
\end{proposition}
\begin{remark}
It is in the proof of the latter proposition that the results contained in \cite{DPR} are used.
\end{remark}
Now we are in place to apply Proposition \ref{Pasqualetto Lucic Rajala} to prove the following:
\begin{theorem}
\label{Chart abs cont}
Let $\varphi:\X\to\R^N$ be a $p$-independent weak chart over a Borel set $E$ of positive measure and with $p\geq 1$, then $\mu=\varphi_\sharp(\m_{|E})<<\Leb^N$ and $N\leq\text{dim}_H (E)$.
\end{theorem}

\begin{proof}
For the moment assume $p\in(1,+\infty)$ and without loss of generality assume $E$ to be compact. Thanks to Lemma $\ref{Pullback and cotangent}$ we deduce that $\varphi^*\Lebspace^p_\mu(\T^*\R^N)$ has dimension $N$ over the set $E$, meaning that $\Lebspace^p_\mu(\T^*\R^N)$ has dimension $N$ over the set $\varphi(E)$. Being the latter module top dimensional, by Proposition $\ref{Pasqualetto Lucic Rajala}$ we have that $\mu<<\Leb^N$ which is the first part of the statement. The second part is immediate since if we had $N>\text{dim}_H(E)$ we would get $\mathcal{H}^N(E)=0$ and since the map $\varphi$ is Lipschitz this implies $\mathcal{H}^N(\varphi(E))=\Leb^N(\varphi(E))\leq C\cdot 0=0$, so that by absolute continuity $\mu(\varphi(E))=\m(E)=0$, which is clearly a contradiction.

For the case $p=1$ note that, since the measure $\m$ is finite, we have $|D(v\cdot\varphi)|_1\leq|D(v\cdot\varphi)|_p$ $\m$-a.e. and for every $v\in\mathbb{S}^{N-1}$, meaning that $\varphi$ is also $p$-independent and the same argument applies.
\end{proof}

\begin{remark}
\label{Bdd dimension tangent}
By virtue of the latter theorem one can see that a control on the Hausdorff dimension $\textit{l}$ of a subset $E$ of a metric measure space grants that the dimension of $\Lebspace^p(\T^{*}\X)_{|E}$ is bounded by $\textit{l}$, hence the cotangent module is finite dimensional there. Moreover the proof presented here simplifies the one in \cite{GP162} since there the authors needed to build independent vector fields in $\Lebspace^2(\T\X)$ with $\Lebspace^2(\m)$-integrable divergence  and push them to $\R^n$ keeping them independent and regular: to do so they had to use additional properties of the map $\Prphi_\varphi$ and the bi-Lipschitz regularity of their chart $\varphi$ was essential. Here instead we mainly exploit the properties of $\R^n$. 
\end{remark}

\printbibliography
 
\end{document}